\documentclass[12pt]{amsart}
\usepackage{amsmath,amssymb,amsthm, color, fullpage, lipsum, algpseudocode,algorithm}
\usepackage[hidelinks]{hyperref}
\usepackage{url}
\usepackage[colorinlistoftodos]{todonotes}
\newcommand {\mbb}{\mathbb}

\newcommand {\be}{\begin{equation}}
\newcommand {\ee}{\end{equation}}

\newtheorem {thm}{Theorem}
\newtheorem {cor}[thm]{Corollary}
\newtheorem {lm}[thm]{Lemma}

\newtheorem {prop}[thm]{Proposition}
\newtheorem {prob}[thm]{Problem}

\theoremstyle{definition}

\theoremstyle{definition}

\theoremstyle{definition}

\theoremstyle{definition}

\usepackage{graphicx}

\AtBeginDocument{\setlength\abovedisplayskip{4pt}}
\AtBeginDocument{\setlength\belowdisplayskip{4pt}}
\renewcommand{\frac}{\dfrac}
\allowdisplaybreaks

\begin{document}

\begin{abstract}
We demonstrate that if $p\in\mathbb{R}[x]$ and $p$ is not an even function, then $\{p(k)\}_{k=0}^{\infty}$ is not a multiplier sequence for the basis of Chebyshev polynomials of the first kind. We also give a characterization of geometric multiplier sequences for the Chebyshev basis. 
\end{abstract}

\title{Parity of Polynomial Multiplier Sequences for the Chebyshev Basis}
\author{Andrzej Piotrowski and Joshua Shterenberg}
\maketitle 

\section{Introduction}\label{intro}
A polynomial is \textit{hyperbolic} if it either has only real zeros or if it is identically zero. An operator $K$ on the vector space of real polynomials $\mathbb{R}[x]$ is called \textit{hyperbolicity preserving} if $K(p)$ is hyperbolic whenever $p$ is hyperbolic. Let $B=\{b_k(x)\}_{k=0}^{\infty}$ be a basis for $\mathbb{R}[x]$. To any sequence of real numbers $\{\gamma_k\}_{k=0}^{\infty}$, we can associate an operator $K$ on the vector space $\mbb{R}[x]$ defined by 
\begin{equation}\label{cms}
{K}\left[\sum_{k=0}^{n} a_kb_k(x)\right]=\sum_{k=0}^{n}\gamma_ka_kb_k(x).
\end{equation}
If the operator (\ref{cms}) is hyperbolicity preserving, then the sequence $\{\gamma_k\}_{k=0}^{\infty}$ is called a \textit{multiplier sequence for the basis $B$}, or a \textit{$B$-multiplier sequence}. 

The characterization of multiplier sequences for various orthogonal bases for $\mathbb{R}[x]$ and the research thereof has been of particular interest within the last many years. It appears this line of research was first developed by Tur\'an (see \cite{Turan} and \cite{Turanstrips}), who used this theory to investigate the famous Riemann hypothesis. For this reason, among others, the theory of multiplier sequences is being developed to this day. In this work, we focus our efforts on multiplier sequences for the Chebyshev polynomials of the first kind (we will simply refer to them as Chebyshev polynomials), $\{T_n\}_{n=0}^{\infty}$,  which can be defined as the unique polynomials satisfying $T_n(\cos(\theta))=\cos(n\theta)$ for $n\in\mathbb{N}\cup\{0\}$.

\section{Related Works}
This paper serves as a continuation of several works which analyze the multiplier sequences of Jacobi polynomials and/or special cases of these, such as the Legendre polynomials. \cite{B} contains a characterization of quadratic Jacobi (and, thus, Chebyshev) multiplier sequences.  \cite{BDFU} contains an analysis of linear, quadratic, and geometric multiplier sequences for the Legendre polynomials. \cite{nreup} involves the study of complex zero decreasing operators and demonstrates the existence of a class of multiplier sequences for the Chebyshev basis. \cite{PILMS} contains a proof that polynomial multiplier sequences for the Legendre basis must have the form $\{h(k^2+k)\}_{k=0}^{\infty}$, where $h\in\mathbb{R}[x]$. \cite{FHMS} contains a proof of the nonexistence of cubic multiplier sequences for the Legendre basis. \cite{FTW} looks more broadly at bases of simple sets and investigates when a set of multiplier sequences for one basis will be included in the set of multiplier sequences for another basis. In \cite{Y}, it is shown that there are no linear  multiplier sequences for the Jacobi basis. This collection of works provides a context for our research, as well as a standard of proofs and approaches to use, which we apply to the Chebyshev basis.   

\section{Computational Tools}
Mathematica, Sage, Wolfram Alpha, and the Python programming language were the central software elements used as investigative tools in this project. They were employed to discover and verify conjectures that led to our main results. All programming was written, compiled, and executed in a local development environment, but is public to view and to download on GitHub: \color{blue}\url{https://github.com/joshshterenberg/Chebyshev_MS}\color{black}.

\section{The symbol of a linear operator }
In this section, we develop a condition that must be met by a hyperbolicity preserving operator that will be useful to us later on. The argument is similar to that in section 3 (``Form and Order'') of \cite{PILMS}, which we outline here for the convenience of the reader. 

For a general linear operator $K$, the symbol of $K$ is given by 
\[G_{K}(x,y)= \sum_{k=0}^{\infty} \frac{(-1)^k   K[x^k] y^k}{k!} =   K[e^{-x y}], 
\]
where the operator $  K$ acts on $e^{-xy}$ as a function of $x$ alone (see \cite{bb} for a comprehensive treatment of the symbol of an operator as it relates to hyperbolicity preserving operators). Furthermore, any linear operator on $\mbb{R}[x]$ can be represented as a differential operator. As in, for example, \cite[Prop. 29]{P}, there is a unique set of real polynomials $\{Q_k\}$ such that
\begin{equation}\label{diffop}
  K = \sum_{k=0}^{\infty} Q_k(x) D^k \qquad \left(D = \frac{d}{dx}\right).
\end{equation}
It follows that the symbol $G_{  K}$ is given by
\begin{equation}\label{GK}
G_{  K}(x,y) =\sum_{k=0}^{\infty} Q_k(x) D^k  e^{-xy}= e^{-xy} \sum_{k=0}^{\infty} Q_k(x) (-1)^k y^k.
\end{equation}
As noted in \cite{BO} and \cite{FHMS}, if the operator $K$ is hyperbolicity preserving then we can act on the function $G_K$ in equation (\ref{GK}) as a function of $x$ alone by the multiplier sequence (for the standard basis) $\{1, 0, 0, 0, \dots\}$ and the resulting function,
\begin{equation}\label{G(0,y)}
G_{  K}(0,y) = \sum_{k=0}^{\infty} Q_k(0) (-1)^k y^k,
\end{equation}
must belong to the Laguerre-P\'olya class, defined as follows: 
The \textit{Laguerre-P\'olya Class}, denoted $\mathcal{L-P}$, is the set of entire functions $\varphi$ which can be expressed in the form
$$\varphi(x) = cx^me^{-ax^2 + bx}\prod_{k=1}^\omega\left(1+\frac{x}{x_k}\right)e^{-x/x_k} \text{\hspace{5 mm}} (0\le\omega\le\infty),$$
where $b,c,x_k \in \mathbb{R}$, $m$ is a non-negative integer, $a\ge 0$, $x_k\neq 0$, and 
$\sum_{k=1}^\omega x_k^{-2}<\infty$. 

To obtain the values of $Q_k(0)$ in equation ($\ref{G(0,y)}$), we use the following lemma which is readily verified by calculating $  K[x^n]$ and evaluating the resulting expression for $x=0$. 
\begin{lm} \label{Q0lem}
For any differential operator
\[  K = \sum_{k=0}^{\infty} Q_k(x)   D^k \qquad \left(D = \frac{d}{dx}\right)\]
on $\mbb{R}[x]$, the coefficient polynomials evaluated at 0 can be computed by 
\[Q_k(0) =  \frac{1}{k!} \left[  K[x^k]\right]_{x=0} \qquad (k=0, 1, 2, \dots).\]
\end{lm}
Thus, we arrive at the following proposition which will be central to our analysis of polynomial sequences of odd degree:

\begin{prop}\label{Glp}
If $  K$ is a linear operator on $\mbb{R}[x]$ and 
\[
G_{  K}(0,y) = \sum_{k=0}^{\infty} Q_k(0) (-1)^k y^k \qquad \left(Q_k(0) =  \frac{1}{k!} \left[  K[x^k]\right]_{x=0}\right)
\]
does not belong to the Laguerre-P\'olya class, then $  K$ is not hyperbolicity preserving.
\end{prop}

\section{Specializing to the Chebyshev Polynomials}
We now specialize the results of the last section to investigate multiplier sequences for the Chebyshev polynomial basis.

\begin{prop}\label{Glem}
Let $\{\gamma_k\}_{k=0}^{\infty}$ be a sequence of real numbers, and let $K$ be the corresponding linear operator on $\mbb{R}[x]$ defined by 
\begin{equation}\label{Tms}
{K}\left[\sum_{k=0}^{n} a_kT_k(x)\right]=\sum_{k=0}^{n}\gamma_ka_kT_k(x),
\end{equation}
where $T_k$ denotes the $k$th Chebyshev polynomial. Let $Q_{2k}(0)$ be defined by 
\begin{equation}\label{Q2k}
Q_{2k}(0) = \frac{2^{1-2k}}{(2k)!}\left[\frac{1}{2}\binom{2k}{k} \gamma_0+\sum_{i=1}^{k}(-1)^i\binom{2k}{k-i}\gamma_{2i}  \right].
\end{equation}
If
\[
G_{  K}(0,y) = \sum_{k=0}^{\infty} Q_{2k}(0) y^{2k}
\]
does not belong to the Laguerre-P\'olya class, then $  K$ is not a hyperbolicity preserving operator. Thus, $\{\gamma_k\}_{k=0}^{\infty}$ is not a multiplier sequence for the Chebyshev basis.
\end{prop}

\begin{proof}
Consider a sequence $\{\gamma_k\}_{k=0}^{\infty}$ with a corresponding operator $K$ given by ($\ref{Tms}$). Applying Lemma \ref{Q0lem}, we have
\[
Q_n(0) = \frac{1}{n!}\left[  K[x^n]\right]_{x=0}.
\]
From the known expansion (see \cite[p. 412]{Cody} or \cite[p. 22]{MH})
\[
x^n = 2^{1-n}\left(\sideset{}{'}\sum_{\substack{j=0 \\ j\equiv n (\operatorname{mod} 2)}}^{n}\binom{n}{\frac{n-j}{2}} T_j(x)\right),
\]
where the prime on the sum means that the contribution of the $j=0$ term (if any) is halved, we obtain 
\begin{align*}
    Q_n(0)&=\frac{1}{n!}\left[  K\left[2^{1-n}\left(\sideset{}{'}\sum_{\substack{j=0 \\ j\equiv n (\operatorname{mod} 2)}}^{n}\binom{n}{\frac{n-j}{2}} T_j(x)\right)\right]\right]_{x=0} \\
    &=\frac{2^{1-n}}{n!}\left(\sideset{}{'}\sum_{\substack{j=0 \\ j\equiv n (\operatorname{mod} 2)}}^{n}\binom{n}{\frac{n-j}{2}} \gamma_j T_j(0)\right). 
\end{align*}
Using the fact (see, e.g., \cite[p. 24]{MH}) that 
\[
T_n(0) = 
\begin{cases}
0  & n \text{ odd,} \\
(-1)^{n/2}  & n \text{ even,} \\
\end{cases}
\]
we see that $Q_n(0) = 0$ when $n$ is odd. If $n$ is even, say $n=2k$, then 
\[
Q_{2k}(0) = \frac{2^{1-2k}}{(2k)!}\left[\frac{1}{2}\binom{2k}{k} \gamma_0+\sum_{i=1}^{k}(-1)^i\binom{2k}{k-i}\gamma_{2i}  \right].
\]
The result now directly follows from Proposition \ref{Glp}.
\end{proof}

The leads to the following useful result. 
\begin{cor}\label{coralt}
With the notation of Proposition \ref{Glem}, if there exists $n\in\mathbb{N}$ such that $Q_{2n}(0)Q_{2n+2}(0)>0$, then $\{\gamma_k\}_{k=0}^{\infty}$ is not a multiplier sequence for the Chebyshev basis.
\end{cor} 
\begin{proof}
The function $G_{  K}(0,y)$ in the previous lemma is an even function of $y$ and an even function in the Laguerre-P\'olya class must have coefficients that alternate in sign (see, for example, \cite[Theorem 3.5]{CCsurvey}). 
\end{proof}

\subsection{Polynomially Interpolated Multiplier Sequences}
We now focus on sequences of the form $\{p(k)\}_{k=0}^{\infty},$ where $p\in\mbb{R}[x]$. This section contains three subsections regarding the quantities $Q_{2k}(0)$ in equation (\ref{Q2k}). The first deals with the contribution of any even powers. The second deals with the contribution of any odd powers. The third combines these for a general polynomial sequence and culminates in our main result. 

\subsubsection{Even powers}\label{subsubsectioneven}
In this section, we appeal to the Chebyshev differential equation to get information on the contribution of even powers to the quantities $Q_{2k}(0)$. 
\begin{prop}
Let $n\in\mathbb{N}\cup\{0\}$ be an even integer and let $K$ be the linear operator on $\mathbb{R}[x]$ defined by
\[
{K}\left[\sum_{k=0}^{n} a_kT_k(x)\right]=\sum_{k=0}^{n} k^na_kT_k(x).
\]
Then the differential operator representation for $K$ in equation (\ref{diffop}) is of finite order. That is to say, $Q_k(x)$ is identically zero for all $k$ sufficiently large.   
\end{prop}

\begin{proof}
Suppose $n=2j$ for some $j\in\mathbb{N}\cup\{0\}$. Let 
\[
p(x) = \sum_{k=0}^{n} a_kT_k(x)\in\mathbb{R}[x].
\]
Since the Chebyshev polynomials satisfy the differential equation (see, e.g., \cite[p. 258 and p. 301]{R} or \cite[p. 85]{MH})
\[
k^2 T_k(x) = x T_k'(x)+ (x^2-1)T_k''(x) \qquad \left(k\in\mathbb{N}\cup\{0\}\right),
\]
we have
\[
K[p(x)] =  \sum_{k=0}^{n} k^{2j}a_kT_k(x) =  \sum_{k=0}^{n} (k^{2})^ja_kT_k(x) = (xD+ (x^2-1)D^2)^j p(x).
\]
Thus, $K =(xD+ (x^2-1)D^2)^j,$ which is a finite order differential operator.
\end{proof}
In general, any even power of $k$ in a sequence of the form 
\[
\{\gamma_k\}_{k=0}^{\infty} = \left\{a_0+a_1 k + a_2 k^2+ \cdots + a_nk^n\right\}_{k=0}^{\infty}
\]
will only contribute a finite number of terms to the differential operator representation (\ref{diffop}) for the corresponding operator defined by (\ref{Tms}). Since our analysis will focus on the behavior of the sequence $Q_{2k}(0)$ for large indices, the even powered terms will not be relevant. 

\subsubsection{Odd powers}

We begin by examining $Q_{2k}(0)$ for $\gamma_k = k^{n}$ where $n$ a positive integer and we will eventually specialize to the case where $n$ is odd. 

\begin{lm}\label{nQA} Let $n\in\mathbb{N}$ and let $\{\gamma_k\}_{k=0}^{\infty} = \{k^n\}_{k=0}^{\infty}$. Then the quantities $Q_{2k}(0)$ defined in equation (\ref{Q2k}) are given by 
\[
Q_{2k}(0) = \frac{2^{1-2k}}{(2k)!}A(n,k),
\]
where
\begin{equation}\label{A(n,k)}
A(n,k)= \sum_{i=1}^{k}(-1)^i\binom{2k}{k-i}(2i)^{n}.
\end{equation}
\end{lm}
\begin{proof}
 Note that, since $n\in\mathbb{N}$, we have that $\gamma_0 = 0.$ The result now follows directly from equation (\ref{Q2k}) and our choice of $\gamma_k$.
\end{proof}

Computer algebra systems give the following simplifications for these expressions:
\begin{align*}
    A(1,k) &= - \frac{k+1}{2k-1} \binom{2k}{k-1}\\[1mm]
    A(3,k) &= \frac{4k(k+1)}{(2k-1)(2k-3)}\binom{2k}{k-1}\\[1mm]
    A(5,k) &= \frac{-16k(k+1)(4k-1)}{(2k-1)(2k-3)(2k-5)} \binom{2k}{k-1}\\[1mm]
\end{align*}
Through a series of lemmas and an auxiliary function, we will show that, in general, $A(n,k)/\binom{2k}{k-1}$ is a rational function of $k$ which is not identically zero whenever $n$ is odd. We will make frequent use of the rising factorial, falling factorial, and hypergeometric function which are defined as follows:
\\\\The \textit{rising and falling factorials} are defined, respectively, by
\[x^{(n)}=\prod_{k=0}^{n-1}(x+k)\qquad\text{and} \qquad (x)_n=\prod_{k=0}^{n-1}(x-k) \qquad (n\in\mathbb{N}),\]
and $x^{(0)}= (x)_0 = 1.$ The \textit{hypergeometric function} is defined in terms of the rising factorial by
\[{}_2F_1(a,b;c;z)=\sum_{n=0}^{\infty}\frac{a^{(n)}b^{(n)}}{c^{(n)}}\frac{z^n}{n!}.\]
\begin{lm}\label{lemsum} For any integer $k>0$, the function 
\[
f(x) = \sum_{i=1}^k  \binom{2k}{k-i} x^i 
\]
can be represented using the hypergeometric function ${}_2F_1$ as 
\[
f(x)= x\cdot\binom{2k}{k-1}\cdot{}_2F_1(1,1-k;2+k;-x).
\]

\end{lm}
\begin{proof} 
Using the definition of the hypergeometric function, re-indexing the sum, and using the fact that $(1-k)^{(i-1)}=0$ for $i>k$, we get 
\begin{align*}
    x\cdot\binom{2k}{k-1}\cdot{}_2F_1(1,1-k;2+k;-x)&=\frac{x(2k)!}{(k-1)!(k+1)!}\sum_{i=0}^{\infty}\frac{1^{(i)}(1-k)^{(i)}}{(2+k)^{(i)}}\frac{(-x)^i}{i!}\\
    &=\frac{(2k)!}{(k-1)!(k+1)!}\sum_{i=1}^{k}\frac{(1-k)^{(i-1)}(-1)^{i-1}}{(2+k)^{(i-1)}}x^{i}.\\
\end{align*}   
Simplifying the coefficient of $x^i$, 
 \begin{align*}   
    \frac{(1-k)^{(i-1)}(-1)^{i-1}}{(2+k)^{(i-1)}}&=\frac{(1-k)(2-k)(3-k)\dots(i-1-k)(-1)^{i-1}}{(2+k)(3+k)(4+k)\dots(i+k)}\\
    &=\frac{(k-1)(k-2)(k-3)\dots(k-i+1)}{(2+k)(3+k)(4+k)\dots(i+k)}.\\
\end{align*}   
Dividing by $(k-1)!$ we have
 \begin{align*} 
    \frac{1}{(k-1)!}\frac{(1-k)^{(i-1)}(-1)^{i-1}}{(2+k)^{(i-1)}}&=\frac{1}{(2+k)(3+k)(4+k)\dots(i+k)\cdot(k-i)!}=\frac{(k+1)!}{(k+i)!(k-i)!}.\\
\end{align*}   
Dividing by $(k+1)!$ then gives 
 \begin{align*} 
    \frac{1}{(k-1)!(k+1)!}\frac{(1-k)^{(i-1)}(-1)^{i-1}}{(2+k)^{(i-1)}}&=\frac{1}{(k+i)!(k-i)!}.\\
\end{align*}   
Therefore, 
 \begin{align*} 
    x\cdot\binom{2k}{k-1}\cdot{}_2F_1(1,1-k;2+k;-x)&= (2k)!\cdot\sum_{i=1}^{k}\frac{1}{(k-1)!(k+1)!}\frac{(1-k)^{(i-1)}\color{black}(-1)^{i-1}}{(2+k)^{(i-1)}\color{black}}x^i\\
    &=\sum_{i=1}^{k}\frac{(2k)!}{(k+i)!(k-i)!}x^i=\sum_{i=1}^{k}\binom{2k}{k-i}x^i.
\end{align*}
\end{proof}
The function in the previous lemma is related to $A(n,k)$ through differentiation. We record this in the following lemma.
\begin{lm}
Let $k\in\mathbb{N}$ and let 
\[
f(x) = \sum_{i=1}^k  \binom{2k}{k-i} x^i . 
\]
The quantity $A(n,k)$ from equation (\ref{A(n,k)}) is given by
\[
A(n,k) = 2^n \left[ \theta^n f(x)\right]_{x=-1} \qquad \left(\theta=xD = x\frac{d}{dx}\right).
\]
\end{lm}
\begin{proof}
From the relation $ \theta^n x^i = i^n x^i$, and the linearity of the operator $  \theta$, we obtain
\begin{align*}
    2^n\left[ \theta^n f(x)\right]_{x=-1}&= 2^n\left[ \theta^n \sum_{i=1}^k  \binom{2k}{k-i} x^i \right]_{x=-1}= 2^n\left[\sum_{i=1}^k  \binom{2k}{k-i} i^n x^i \right]_{x=-1} \\
    &= 2^n\sum_{i=1}^k  \binom{2k}{k-i} i^n (-1)^i = A(n,k).\\
\end{align*}
\end{proof}

Combining the previous two lemmas, we obtain:

\begin{cor} \label{Afun} For positive integers $n$ and $k$, 
\[A(n,k) = 2^n\binom{2k}{k-1}\left[ \theta^n\left(x\cdot{}_2F_1(1,1-k;2+k;-x)\right)\right]_{x=-1}\qquad \left( \theta = x\frac{d}{dx}\right).\]
\end{cor}

Now, in order to simplify the derivatives in the previous corollary, we define a new function of the variable $x$ with parameters $n$ and $k$ by
\begin{equation}\label{g}
g(n,k;x) = x^{n+1}\cdot{}_2F_1(1+n,1+n-k;2+n+k;-x)
\end{equation}
and note that the conclusion of Corollary \ref{Afun} can be stated as 
\begin{equation}\label{Ag}
A(n,k) = 2^n\binom{2k}{k-1}\left[ \theta^n g(0,k;x)\right]_{x=-1}.
\end{equation}
When calculating a related expression for the first few values of $n$, we obtained:
\[
 \theta g(0,k;x) = g(0,k;x)+ \frac{k-1}{k+2} g(1,k;x)
\]
\[
 \theta g(1,k;x) = 2g(1,k;x)+ 2\frac{k-2}{k+3} g(2,k;x) 
\]
\[
 \theta g(2,k;x) = 3g(2,k;x)+ 3\frac{k-3}{k+4} g(3,k;x) 
\]
and this pattern generalizes as shown in the following lemma. As the arguments of $k$ and $x$ do not change, we will write $g(n) = g(n,k;x)$ for clarity. 
\begin{lm}\label{lmgd}
Let 
\[
g(n) = x^{n+1}\cdot{}_2F_1(1+n,1+n-k;2+n+k;-x).
\]
Then 
\[
{ \theta}g(n) = (n+1)\left(g(n) + \frac{k-n-1}{k+n+2}g(n+1)\right) \qquad \left(\theta = x\frac{d}{dx}\right).
\]
\end{lm}
\begin{proof}
We calculate
\begin{align*}
    { \theta}g(n)
    &=x \frac{d}{dx}\left[x^{n+1}\cdot{}_2F_1(1+n,1+n-k;2+n+k;-x)\right] \\
    &= (n+1)x^{n+1}{}_2F_1(1+n,1+n-k;2+n+k;-x) \\
    &{} \qquad \qquad  + x^{n+2}\frac{d}{dx}\left({}_2F_1(1+n,1+n-k;2+n+k;-x)\right)\\
    &= (n+1)g(n) - x^{n+2} \frac{(1+n)(1+n-k)}{2+n+k}{}_2F_1(2+n, 2+n-k; 3+n+k; -x)\\
    &=(n+1)\left(g(n) + \frac{k-n-1}{2+n+k}g(n+1)\right),\\
\end{align*}
where we have used the identity (see, e.g., \cite[p. 69]{R})
$$
\frac{d}{dx} F(a, b;c;x) = \frac{ab}{c} F(a+1,b+1;c+1;x).
$$
\end{proof}

Using Lemma \ref{lmgd} to recursively calculate 
$ \theta^ng(n,k,1)$, we obtained constants for each term which appear in the Online Encyclopedia of Integer Sequences \cite[A028246]{OEIS} as Worpitzky's triangular array involving alternating binomial power sums. From this we obtained the following lemma. Again, we will suppress the $k$ and $x$ from the function $g$ for clarity.
\begin{lm}\label{dg0}
For positive integers $n$ and $k$, we have
\[ \theta^ng(0)=\sum_{i=0}^{n}\frac{(k-1)_i}{(2+k)^{(i)}}g(i)\cdot C_{i,n}\qquad \left(\theta = x\frac{d}{dx}\right),\]
where $g(n) = g(n,k;x)$ is defined in equation \emph{(\ref{g})} and $C_{i,n}$ is Worpitzky's triangular array defined by
\begin{equation}\label{C_{i,n}}
C_{i,n}=\frac{1}{i+1}\sum_{j=0}^{i+1}(-1)^{i-j+1}\binom{i+1}{j}j^{n+1}.
\end{equation}
\end{lm}
\begin{proof}
(By induction.) The result is readily verified for $n=0$. Furthermore, the case $n=1$ can be obtained directly from Lemma \ref{lmgd}. Now suppose the lemma holds for some integer $n\geq 0$. Then 
\begin{align*}
     \theta^{n+1}g(0)
    &= \theta \left(\sum_{i=0}^{n}\frac{(k-1)_i}{(2+k)^{(i)}}g(i)\cdot C_{i,n}\right)\\
    &=\left(\sum_{i=0}^{n}\frac{(k-1)_i}{(2+k)^{(i)}}  \theta g(i)\cdot C_{i,n}\right)\\
    &=\sum_{i=0}^{n}\frac{(k-1)_i}{(2+k)^{(i)}}(i+1)g(i)\cdot C_{i,n} +\sum_{i=0}^{n}\frac{(k-1)_i}{(2+k)^{(i)}} (i+1)\frac{k-i-1}{k+i+2}g(i+1)\cdot C_{i,n}\\
    &=\sum_{i=0}^{n}\frac{(k-1)_i}{(2+k)^{(i)}}(i+1) g(i)\cdot C_{i,n} +\sum_{i=1}^{n+1}\frac{(k-1)_{i}}{(2+k)^{(i)}} g(i)\cdot iC_{i-1,n}\\
    &=\sum_{i=0}^{n+1}\frac{(k-1)_i}{(2+k)^{(i)}} g(i)\cdot C_{i,n+1}\\
\end{align*}
where, for $0\leq i<n+1$ we have used the identities (see, e.g., \cite[p. 210]{W})
\[
C_{i, n+1} = (i+1)C_{i,n}+ iC_{i-1, n}
\]
and 
\[
C_{n+1,n+1} = (n+1)C_{n,n}.
\]
\end{proof}

We will need to evaluate these expressions at $x=-1.$ The next lemma aids us in this endeavour.
\begin{lm} \label{g-1}
Let $i,k\in\mathbb{N}\cup\{0\}.$ If $k>i/2$, then 
\[
g(i,k;-1) = (-1)^{i+1}\frac{(k+1)^{(i+1)}}{(2k)_{i+1}}.
\]
\end{lm}
\begin{proof}
By definition (see equation (\ref{g}))
\[
g(i,k;-1) = (-1)^{i+1}{}_2F_1(1+i, 1+i-k; 2+i+k; 1).
\]  
Using the formula (\cite[p. 49, l. 7]{R})
$$
F(a,b;c;1) = \frac{\Gamma(c) \Gamma(c-a-b)}{\Gamma(c-a) \Gamma(c-b)} \qquad (\text{Re}(c-a-b)>0)
$$    
we obtain for $k>i/2$
\begin{align*} 
    g(i,k;-1)&=  (-1)^{i+1}\frac{\Gamma(2+i+k)\Gamma(2k-i)}{\Gamma(1+k)\Gamma(2k+1)}\\
    &=  (-1)^{i+1}\frac{(1+i+k)!(2k-i-1)!}{k!(2k)!}\\
    &= (-1)^{i+1}\frac{(k+1)^{(i+1)}}{(2k)_{i+1}}.\\
\end{align*}
\end{proof}
We may now combine the preceding results to obtain the following proposition.
\begin{prop} \label{propQ}For positive integers $k$ and $n$ satisfying $k>n/2$, the expressions $A(n,k)$ defined in equation \emph{(\ref{A(n,k)})} are given by
\[A(n,k) = 2^n\binom{2k}{k-1}(k+1)\frac{N(n,k)}{(2k)_{n+1}},\]
where
\[
N(n,k) = \sum_{i=0}^{n}(-1)^{i+1} C_{i,n}(k-1)_{i}(2k-i-1)_{n-i},
\]
and
$C_{i,n}$ is defined as in equation \emph{(\ref{C_{i,n}})}. In particular, we note that $A(n,k)/\binom{2k}{k-1}$ is a rational function of the variable $k$.
\end{prop}
\begin{proof}
Combining equation (\ref{Ag}), Lemma \ref{dg0}, and Lemma \ref{g-1}, we obtain
\begin{align*}
A(n,k) &= 2^n\binom{2k}{k-1}\left[ \theta^n g(0,k;x)\right]_{x=-1}\\
&=2^n\binom{2k}{k-1}\sum_{i=0}^{n}\frac{(k-1)_i}{(2+k)^{(i)}}g(i,k;-1)\cdot C_{i,n}\\
&=2^n\binom{2k}{k-1}\sum_{i=0}^{n}\frac{(k-1)_i}{(2+k)^{(i)}}(-1)^{i+1}\frac{(k+1)^{(i+1)}}{(2k)_{i+1}}\cdot C_{i,n}\\
&=2^n\binom{2k}{k-1}\sum_{i=0}^{n}(-1)^{i+1}\frac{(k-1)_i(k+1)}{(2k)_{i+1}}\cdot C_{i,n}.
\end{align*}
Finally, the result follows from rewriting the sum as a single fraction using the common denominator $(2k)_{n+1}$.
\end{proof}

We summarize this section in the following corollary.

\begin{cor}\label{corQ}
Let $n$ be a positive integer and, for $k\in\mathbb{N}\cup\{0\}$, define $\gamma_k=k^n$. Then, for $k>n/2,$ the quantities $Q_{2k}(0)$ defined in equation (\ref{Q2k}) are given by 
\[
Q_{2k}(0) = \frac{2^{1-2k+n}}{(2k)!}\binom{2k}{k-1}(k+1)\frac{N(n,k)}{(2k)_{n+1}},\]
where
\[
N(n,k) = \sum_{i=0}^{n}(-1)^{i+1} C_{i,n}(k-1)_{i}(2k-i-1)_{n-i},
\]
and
$C_{i,n}$ is defined as in equation \emph{(\ref{C_{i,n}})}.
\end{cor}
\begin{proof}
Combine Lemma \ref{nQA} with Proposition \ref{propQ}.
\end{proof}

We note that, in light of section \ref{subsubsectioneven}, when $n$ is even, the quantities $Q_{2k}(0)$ must be zero for all sufficiently large $k$. Thus, the rational function that appears must reduce to the identically zero function, i.e., $N(n,k)$ must be identically zero. When $n$ is odd, this is not the case. We record this in the following lemma that will be used in the main result of the next section.
\begin{lm}\label{N(n,n/2)}
If $n$ is odd, then $N(n,n/2)\neq 0$, where $N$ is defined in Corollary \ref{corQ}.
\end{lm}

\begin{proof}
We have
\[N(n,n/2) = (-1)^{n+1} C_{n,n}\left(\frac{n}{2}-1\right)\left(\frac{n}{2}-2\right)\left(\frac{n}{2}-3\right)\cdots \left(\frac{n}{2}-n\right).\]
Since (see, e.g., \cite[p. 225]{W}) $C_{n,n} = n!\neq 0$  and $n$ is odd, $N(n,n/2)\neq 0$ (Note, if $n$ is even, then the quantity is zero, as expected).  
\end{proof}

\subsubsection{Main Results}
We are now in a position to state and prove our main result.
\begin{thm}
Suppose 
\[
\{\gamma_k\}_{k=0}^{\infty} = \{k^n+ b_1k^{n-1}+ b_2 k^{n-2}+ \cdots + b_n\}_{k=0}^{\infty} \qquad (n\in\mathbb{N}\cup\{0\}, b_i\in\mathbb{R})
\] 
is a multiplier sequence for the Chebyshev basis. Then $n$ is even and $a_i = 0$ whenever $i$ is odd. I.e., the polynomial interpolating the sequence $\gamma_k$ must have no odd powers. It must be an even function. 
\end{thm}
\begin{proof}
We focus our attention on the sequence $Q_{2k}(0)$ for large values of $k$. As seen in section \ref{subsubsectioneven}, we may omit any even powers in the sequence $\{\gamma_k\}_{k=0}^{\infty}$. What remains is a sequence of the form
\begin{equation*}
\gamma_k = k^{n}+ b_{2}k^{n-2}+ b_4k^{n-4}+ \cdots +b_{n-1} k,
\end{equation*}
where we are assuming that $n$ is odd.
For this sequence, equation (\ref{Q2k}) becomes
\begin{equation*}
Q_{2k}(0) = \frac{2^{1-2k}}{(2k)!}\left[\sum_{i=1}^{k}(-1)^i\binom{2k}{k-i}\gamma_{2i} \right]
\end{equation*}
or, in light of Lemma \ref{nQA},
\[
Q_{2k}(0) = \frac{2^{1-2k}}{(2k)!}\left[A(n,k)+ b_2A(n-2,k)+ b_4A(n-4,k)+ \cdots b_{n-1}A(1,k) \right].
\]
Applying Proposition \ref{propQ}, we obtain
\[
Q_{2k}(0)=\frac{2^{1-2k}}{(2k)!}\binom{2k}{k-1}(k+1)\left[\frac{2^n}{(2k)_{n+1}}N(n,k)+\frac{2^{n-2}b_2}{(2k)_{n-1}}N(n-2,k)+\cdots+\frac{2^{1}b_{n-1}}{(2k)_{1}}N(1,k)\right]
\]
Factoring out $(2k)_{n+1}$, we see that the sign of $Q_{2k}(0)$ will be completely determined by the expression 
\[2^nN(n,k)+2^{n-2}(2k-n)(2k-n+1)b_2N(n-2,k)+\cdots+2^{1}(2k-n)^{(n-1)}b_{n-1}N(1,k),\]
which is a polynomial in the variable $k$. We claim that this polynomial is not identically zero. Indeed, by setting $k=n/2$ the expression reduces to $2^nN(n,n/2)$ which is non-zero by Lemma \ref{N(n,n/2)}. Therefore, for all sufficiently large values of $k$, this polynomial in $k$ must be strictly positive or strictly negative. It follows that the same is true of the sequence $Q_{2k}(0)$. Therefore, the result follows from Corollary \ref{coralt}. 
\end{proof}

We can now draw some conclusions about geometric sequences as well. First, we need two known results. 

\begin{thm}\cite[Theorem 2]{FTW} Let $B = \{b_k(x)\}_{k=0}^{\infty}$ and $Q = \{q_k(x)\}_{k=0}^{\infty}$ be simple sets of polynomials. If there exists an $\alpha>1$ such that both $\{\alpha^k\}_{k=0}^{\infty}$ and $\{\alpha^{-k}\}_{k=0}^{\infty}$ are $B$-multiplier sequences, then every $Q$-multiplier sequence is also a $B$-multiplier sequence. 
\end{thm}

\begin{thm}\cite[Theorem 2]{PILMS} If a multiplier sequence for the Legendre basis can be interpolated by a polynomial $p$, then $p(x) = h(x^2+x)$ for some polynomial $h\in \mathbb{R}[x]$.
\end{thm}

Since there are clearly multiplier sequences for the Legendre basis which are not multiplier sequences for the Chebyshev basis (and vice versa) we can conclude that there does not exist an $\alpha>1$ such that both $\{\alpha^k\}_{k=0}^{\infty}$ and $\{\alpha^{-k}\}_{k=0}^{\infty}$ are multiplier sequences for the Chebyshev basis. In fact, much more is true. We detail this in our next theorem which characterizes geometric multiplier sequences for the Chebyshev basis.
\begin{thm}
Let $r\in\mathbb{R}$. Then $\{r^k\}_{k=0}^{\infty}$ is a multiplier sequence for the Chebyshev basis if and only if $r\in\{-1,0,1\}.$
\end{thm}
\begin{proof}
Let $K$ be the operator corresponding to the sequence $\{r^k\}_{k=0}^{\infty}$ as defined in (\ref{Tms}). A calculation shows that 
\begin{align*}
K[(x+r/3)^3] &=K\left[\frac{1}{4}T_3(x)+ \frac{r}{2} T_2(x)+ \left(\frac{3}{4}+ \frac{r^2}{3}\right)T_1(x)+ \left(\frac{r}{2}+ \frac{r^3}{27}\right)T_0(x)\right]\\
&=\frac{r^3}{4}T_3(x)+ \frac{r^3}{2}T_2(x)+ \left(\frac{3r}{4}+ \frac{r^3}{3}\right)T_1(x)+ \left(\frac{r}{2}+ \frac{r^3}{27}\right)T_0(x)\\
&= r^3 x^3+ r^3 x^2 + \left(\frac{3}{4}r - \frac{5}{12}r^3\right)x +\frac{1}{2}r - \frac{25}{54}r^3.
\end{align*}
It is known (see, e.g., \cite[p. 36-40]{cajori}) that a cubic function $ax^3+bx^2+cx+d$ has non-real zeros if and only if the quantity 
\[
\Delta=b^2 c^2 - 4 a c^3 - 4 b^3 d - 27 a^2 d^2 + 18 a b c d
\]
is negative. For $K[(x+r/3)^3]$, we obtain $\Delta = -27 r^6 (r^2-1)^2/16$. It follows that if $\{r^k\}_{k=0}^{\infty}$ is a multiplier sequence for the Chebyshev basis, then $r\in\{-1,0,1\}$. 

Conversely, the sequences $\{0^k\}_{k=0}^{\infty}$ and $\{1^k\}_{k=0}^{\infty}$ are clearly multiplier sequences for any basis. The sequence $\{(-1)^k\}_{k=0}^{\infty}$ is also a multiplier sequence for the Chebyshev basis as a consequence of the symmetry relation $T_n(-x) = (-1)^n T_n(x)$ (see, e.g., \cite[p. 2-3]{MH}. Indeed, the corresponding operator $K$ in this case satisfies $T[p(x)] = p(-x)$ and so $T[p]$ has the same number of real zeros as $p$.
\end{proof}

\section{Conclusions and Further Work}
We have demonstrated that if a multiplier sequence for the Chebyshev basis can be interpolated by a polynomial, then the polynomial must be an even function. This result closely mirrors one that was obtained for the Legendre polynomials in \cite{PILMS}. The next point of research beyond this could be to analyze the more general Gegenbauer and Jacobi polynomials (see, for example, \cite{R} for the definition of these polynomials) in the same fashion. In particular, based on the results for the Legendre and Chebyshev bases, we pose the following problem.
\begin{prob}
Is it true that if a multiplier sequence for the Jacobi polynomial basis $\{P_k^{(\alpha,\beta)}\}_{k=0}^{\infty}$ can be interpolated by a polynomial $p$, then there must exist $h\in\mbb{R}[x]$ such that $p(x) = h(x(x+\alpha+\beta+1))$?
\end{prob}
One could also attempt to determine which even polynomials $p$ generate a multiplier sequence for the Chebyshev basis. To date, and to the best of our knowledge, only the quadratic multiplier sequences have been characterized for the Chebyshev basis. Finally, one could attempt to find a complete characterization of multiplier sequences for the Jacobi polynomial basis. 

\section{Acknowledgement} The authors are grateful to the anonymous referee for their detailed and helpful comments, and for pointing out that the results of \cite{FTW} were relevant to this work.

\end{document}